\documentclass{tac}

\usepackage{fullpage}
\usepackage{amsmath,amsfonts,amssymb}
\usepackage{hyperref}
\usepackage[all]{xy}

\author{Mat\'\i as Menni}

\thanks{This work was supported by CONICET (Argentina), PIP 11220200100912CO, and by the European Union’s Horizon 2020 research and innovation programme under the Marie Sklodowska-Curie grant agreement No 101007627.}

\address{Conicet and Universidad Nacional de La Plata, Argentina.}

\title{
Bi-directional models of \\ `Radically Synthetic' Differential Geometry
}

\copyrightyear{2024}


\eaddress{matias.menni@gmail.com}

\newcommand{\Set}{\mathbf{Set}}

\newcommand{\Ring}{\mathbf{Ring}}

\newcommand{\Cinfty}{{C^{\infty}}}

\newcommand{\inftyRing}{\Cinfty\textnormal{-}\Ring}

\newcommand{\twopl}[2]{\langle #1, #2\rangle}
\newcommand{\pr}{\mathrm{pr}} 

\newcommand{\Aff}{\mathrm{Aff}}

\newcommand{\psh}[1]{\widehat{#1}} 
\newcommand{\Psh}{\psh}

\newcommand{\calC}{\ensuremath{\mathcal C}}

\newcommand{\calE}{\mathcal{E}} 

\newcommand{\calG}{\mathcal{G}}

\newcommand{\calS}{\ensuremath{\mathcal S}}

\newcommand{\calX}{{\cal X}}

\newcommand{\Real}{\mathbb{R}}
\newcommand{\Complex}{\mathbb{C}}

\newcommand{\final}{\mathbf{1}}   

\newcommand{\aunit}{0} 

\DeclareMathOperator{\Spec}{Spec}
\DeclareMathOperator{\Dec}{Dec}

\begin{document}
\maketitle

\begin{abstract}
The radically synthetic foundation for smooth geometry formulated in \cite{Lawvere2011} postulates a space $T$ with the property that it has a unique point and, out of  the   monoid  $T^T$ of endomorphisms,  it  extracts a submonoid $R$ which, in many cases, is the (commutative) multiplication of a rig structure.  The rig $R$ is said to be {\em  bi-directional} if  its subobject of invertible elements has two connected components. In this case, $R$ may be equipped with a  pre-order compatible with the rig structure. 
We adjust the construction of `well-adapted' models of Synthetic Differential Geometry  in order to build the first pre-cohesive toposes with a bi-directional $R$.
We also show that, in one of these pre-cohesive variants,  the pre-order on $R$, {\em derived} radically synthetically  from bi-directionality, coincides with that {\em defined} in the original model. 
\end{abstract}

\tableofcontents

\section{Introduction}

The origin of  Synthetic Differential Geometry (SDG) may be traced back to certain  1967 lectures by Lawvere, later summarized in \cite{Lawvere79}.
That summary includes, between brackets, some remarks based on developments that  occurred since the original lectures.
It  postulates a locally catesian closed category $\calX$, a ring $R$ therein, and an isomorphism ${R^D \cong R \times R}$, where ${D \rightarrow R}$ denotes the subobject of elements of square $0$. Between brackets, it is observed that  such isomorphism could be interpreted to mean that the canonical ${R \times R \rightarrow R^D}$ is invertible, as had been done in several papers by Kock. See \cite{Kock1977} where this interpretation  is introduced and where rings satisfying the resulting `Kock-Lawvere (KL)' axiom are called {\em of line type}.

The summary   also sketches the construction of  models, including the role of the algebraic theories of real-analytic and of $\Cinfty$ functions. Between brackets, it is mentioned that, in 1978, Dubuc succeeded in constructing a model of the axioms containing the category of real $\Cinfty$-manifolds. See \cite{Dubuc1979} where $\Cinfty$-rings play the prominent role.

The first book on SDG  appeared in 1981 and it was reprinted as a second edition in \cite{KockSDG2ed}.
Several other books have appeared since 1990 giving respective accounts of the development of the subject 
  \cite{MoerdijkReyesBook, LavendhommeSDG, BellSDG, Kock2010} and \cite{BungeEtAl2018}.

Much can be done with the  KL-axiom alone but, for some purposes,  (e.g. integration), a pre-order on a ring of line type $R$, compatible with the ring structure, is also postulated. 
The typical models are  `gros' toposes such as  those arising in classical Algebraic Geometry \cite{KockSDG2ed} and those intentionally built to produce `well-adapted' models embedding the category of manifolds \cite{Dubuc1979,KockSDG2ed, MoerdijkReyesBook}. See also \cite{Menni2021a} for more recent examples, similar to those in Algebraic Geometry over $\Complex$, but over the simple rig with idempotent addition.

The `gros' vs `petit' distinction among toposes appears already in \cite{SGA4} but the idea to axiomatize toposes `of spaces' is from the early `80s  \cite{Lawvere86}. See also \cite{Lawvere07} for a more recent formulation and \cite{MenniExercise} or \cite{Menni2021} for the definition of {\em pre-cohesive} geometric morphism. For instance, the `gros' Zariski topos $\calE$ determined by the field $\Complex$ of complex numbers is a well-known model of SDG and the canonical  geometric morphism ${ \calE \rightarrow \Set}$ is pre-cohesive. The same holds for some of the models in \cite{Menni2021a}.
On the other hand, the  well-adapted models of SDG are intuitively  toposes of spaces  but, as far as I know, very little work  has been done to  relate them with Axiomatic Cohesion.
Perhaps the only exception are the results in \cite{MoerdijkReyesBook} noting that the well-adapted models discussed there are local, which is one of the requirements in \cite{Lawvere86} (and also in the definition of pre-cohesive map).
Another requirement is that  the inverse image of a pre-cohesive geometric morphism ${ \calE \rightarrow \calS}$ must  have a finite-product preserving left adjoint ${ \calE \rightarrow \calS}$ which, intuitively, sends a space $X$ to the associated set ${\pi_0 X}$ of connected components of $X$. This cannot be true  for the canonical geometric morphism from the  `smooth Zariski topos'  \cite[VI]{MoerdijkReyesBook} to the topos of sets because, in constrast with the classical cases over a field, where every affine scheme is a finite coproduct of connected ones, there are affine $\Cinfty$-schemes with infinite coproduct decompositions. 

The paragraph above does not imply that `well-adaptation' is incompatible with Cohesion. We will show in Sections~\ref{SecCoextensive} and \ref{SecFirstExample} how to modify the techniques to construct well-adapted models of SDG so that they produce very simple (presheaf) pre-cohesive toposes (over $\Set$) with rings of line type.
Moreover, these  pre-cohesive toposes  also model something more radical.

The radically synthetic foundation for smooth geometry formulated  in \cite{Lawvere2011} (and briefly recalled in Section~\ref{SecRSDG}) postulates a space $T$ with the property that it has a unique point and, out of  the  monoid  $T^T$ of endomorphisms of $T$,  it  extracts a submonoid $R$ which, in many cases, is the (commutative) multiplication of a rig structure.  The rig $R$ is said to be {\em  bi-directional} if  its subgroup $U$ of invertible elements is such that ${\pi_0 U = \mathbb{Z}_2}$ the multiplicative group with two elements. In this case, $R$ may be equipped with a  pre-order compatible with the rig structure. 
 As explained loc.cit., this is  `radically synthetic' in the sense that all algebraic structure is derived from constructions on the geometric spaces rather than assumed. (See also \cite{LawvereSDGoutline} which is unpublished but freely available from Lawvere's webpage.)

 We remark that no explicit models of bi-directional  Radical SDG are considered in \cite{Lawvere2011}. It is clear  that the objects $D$ arising in SDG can play the role of $T$ and that they will give the expected result, but the issue of bi-directionality is more subtle. On the one hand, $R$ is not bidirectional in the models coming from Algebraic Geometry; on the other hand, we don't know which of the known well-adapted models of SDG have a `$\pi_0$' functor. Nevertheless, many well-adapted models at least satisfy that the subobject of invertibles in $R$ is not connected; in fact,  in these cases, $U$ is  representable by the $\Cinfty$-ring of smooth $\Real$-valued functions on the non-connected manifold  ${ (-\infty, 0) + (0, \infty)}$  so, at least, they are `bi-directional' in this sense. Essentially the same phenomenon implies that, in our pre-cohesive models, bi-directionality holds in the  sense of the previous paragraph.
 
 Finally we show in Section~\ref{SecW} that,  in one of the pre-cohesive models, the pre-order on $R$, {\em derived} radically synthetically from bi-directionality, coincides with the pre-order that is {\em defined} in the analogous model of SDG.

\section{The coextensive category of $\Cinfty$-rings}
\label{SecCoextensive}

Let $\Ring$ be the coextensive category of rings so that the slice ${\Real/\Ring}$ is the category of $\Real$-algebras.
We next recall the algebraic theory of  $\Cinfty$-rings   \cite[III.5]{KockSDG2ed}. For any finite set $n$, the $n$-ary operations are the smooth functions ${\Real^n \rightarrow \Real}$. The associated algebraic category of $\Cinfty$-rings will be denoted by $\inftyRing$.
There is an evident algebraic functor ${\inftyRing \rightarrow \Real/\Ring}$ that has the following not quite so evident property.

\begin{lemma}[{\cite[Proposition~{III.5.4}]{KockSDG2ed}}]\label{LemIdealsInInftyRings}
Let $A$ be a $\Cinfty$-ring and let ${I \subseteq A}$ be an ideal in the usual ring-theoretic sense. Then the $\Real$-algebra ${A/I}$ carries a unique structure of $\Cinfty$-ring such that the quotient ${A \rightarrow A/I}$ is a morphism of $\Cinfty$-rings. Hence, as map in $\inftyRing$, it is the universal map from $A$ with kernel ${I \subseteq A}$.
\end{lemma}

As an application we prove the following (folk?) basic fact.

\begin{proposition}\label{PropInftyRingIsCoextensive} The category $\inftyRing$ is coextensive.
\end{proposition}
\begin{proof}
We show that coextensivity lifts from ${\Real/\Ring}$ along  ${\inftyRing \rightarrow \Real/\Ring}$.
It is enough to check that this functor preserves and reflects both finite products and pushouts along product projections.
The issue of finite products is easy because any algebraic functor creates limits so  it also preserves and reflects them.

Consider now a product projection  ${\pi : A \rightarrow B}$ in $\inftyRing$.
It is well-known that, as a map in ${\Real/\Ring}$, it may be identified with ${A \rightarrow A/(e)}$ for some idempotent ${e \in A}$.
Then so is the case in $\inftyRing$ by Lemma~\ref{LemIdealsInInftyRings}.
Hence, for any ${f : A \rightarrow C}$ in $\inftyRing$ the square below
\[\xymatrix{
A \ar[d]_-f \ar[r] & A/(e) \ar[d] \\
C \ar[r] & C/(f e)
}\]
is a pushout both in $\inftyRing$ and in ${\Real/\Ring}$.
It is then clear that ${\inftyRing \rightarrow \Real/\Ring}$ both preserves and reflects these pushouts.
\end{proof}

 Let ${(\inftyRing)_{fg} \rightarrow \inftyRing}$ be the full subcategory of finitely generated $\Cinfty$-rings.
 
\begin{corollary}\label{CorLociIsExtensive} 
The essentially small category ${(\inftyRing)_{fg}}$ is coextensive and the full inclusion ${(\inftyRing)_{fg} \rightarrow \inftyRing}$ preserves finite products.
\end{corollary}
\begin{proof}
It is enough to show that the full  subcategory ${(\inftyRing)_{fg} \rightarrow \inftyRing}$ is closed under finite products and direct factors.
Closure under direct factors follows from their description in terms of idempotents as in the proof of Proposition~\ref{PropInftyRingIsCoextensive}.
So it remains to prove that finitely generated $\Cinfty$-rings are closed under binary product.
It is enough to check that ${\Cinfty(\Real^m) \times \Cinfty(\Real^n)}$ is finitely generated for all finite $m, n$.
Using (smooth) Tietze one may show that a simple ${\Cinfty(\Real^{m+n+1}) \rightarrow \Cinfty(\Real^m) \times \Cinfty(\Real^n)}$ is surjective.
\end{proof}

In contrast with Lemma~\ref{LemIdealsInInftyRings} the forgetful ${\inftyRing \rightarrow \Real/\Ring}$ does not create the universal solution to inverting an element. This will be of key importance throughout the paper.

Also, in contrast with the case of $k$-algebras for a field $k$: it is not the case that every finitely generated $\Cinfty$-ring is finitely presentable \cite[Example~{III.5.5}]{KockSDG2ed}, and  it is not the case that every finitely generated $\Cinfty$-ring is  a finite direct product of directly indecomposable ones.
Moreover, by the Nullstellensatz,  every non-final finitely generated $\Complex$-algebra  has a (co)point, but see \cite[p.~{165}]{KockSDG2ed} for an example of a non-final finitely generated $\Cinfty$-ring without points.

Let $\Aff_{\Cinfty}$ be the (extensive)  opposite of the category  of finitely generated $\Cinfty$-rings.
 Its objects might be called {\em affine $\Cinfty$-schemes}.

\section{Radically Synthetic Differential Geometry}
\label{SecRSDG}

Euler's observation that real numbers are ratios of infinitesimals is the topic of \cite{Lawvere2011}. 
To make that observation rigorous, Lawvere suggests some basic properties of the underlying category of spaces  and postulates the existence of an object  `of infinitesimals' with the sole property that it has a unique point. Then he shows how  some simple axioms of geometric nature  allow to construct a pre-ordered ring `of Euler reals'. In this section, we briefly recall some of these ideas and relate them to SDG.

Let $\calE$ be an extensive category with finite limits  and let ${0 : \final \rightarrow T}$ be a pointed object in $\calE$ such that $T$ is exponentiable.
For any object $X$ in $\calE$, ${X^T}$ is called the {\em tangent bundle} of the space $X$, with  evaluation at $0$ inducing the bundle map ${ev_0 : X^T \rightarrow X}$.

\begin{definition}\label{DefEulerReals}
{\em
The subobject ${R \rightarrow T^T}$ of {\em Euler reals} is defined by declaring
$$\xymatrix{
R \ar[d] \ar[r] & T^T \ar[d]^-{ev_0} \\
\final \ar[r]_-0 & T
}$$
to be a pullback in $\calE$.
}
\end{definition}

The exponential $T^T$ carries a  canonical monoid structure determined by composition in $\calE$. This monoid structure restricts to $R$ and is called {\em multiplication}. The transposition ${\final \rightarrow T^T}$ of the evident composite ${T \rightarrow \final \rightarrow T}$ factors as a map ${0 : \final \rightarrow R}$ followed by the inclusion ${R \rightarrow T^T}$. 
In this way, $R$ has the intrinsic structure of a monoid with $0$.

We are mainly interested in categories of spaces that are toposes, but some of the ideas may be directly illustrated at the level of sites.

\begin{example}[Affine $k$-schemes.]\label{ExAffineKschemes}
Let $k$ be a field and let $\Aff_k$ be the (cartesian and extensive) opposite of the category of finitely presentable $k$-algebras.
If $A$ is one such then ${\Spec A}$ denotes the corresponding object in $\Aff_k$.
In particular, let  ${k[\epsilon] = k[y]/(y^2)}$ and   let  ${T =   \Spec(k[\epsilon])}$. Clearly $T$ has a unique point and a direct calculation shows that   ${T^T = \Spec (k[\varepsilon, x]/(\epsilon x)) = \Spec(k[x,y]/(x y, y^2))}$ and  that  ${R = \Spec(k[x])}$, the affine line.
The subobject ${R \rightarrow T^T}$  corresponds to the unique map ${k[\varepsilon, x] \rightarrow k[x]}$  that sends $x$ to $x$ and $\epsilon$ to $0$.
Also, the multiplication morphism ${R\times R \rightarrow R}$ in $\Aff_k$ corresponds to the morphism ${k[x] \rightarrow  k[x]\otimes_k k[x] = k[y, z]}$ that sends ${x}$ to ${y  z}$. It is well-known that this multiplication is the multiplicative structure of a ring of line type.
\end{example}

The preservation properties of the Yoneda functor imply that, in the topos of presheaves on $\Aff_k$, the monoid of Euler reals determined by the presheaf representable by $T$ is representable by $R$. On the other hand, ${\Aff_k \rightarrow \Psh{\Aff_k}}$ does not preserve finite coproducts.

\begin{example}[The Gaeta topos determined by field.]\label{ExKgaeta} 
Still assuming that $k$ is a field, extensivity of $\Aff_k$ allows us to consider the Gaeta subtopos ${\calG_k \rightarrow \Psh{\Aff_k}}$ of finite-product preserving presheaves. The Yoneda embedding factors through this subtopos and that the factorization ${\Aff_k \rightarrow \calG_k}$ preserves not only limits but also finite coproducts. As in the presheaf case, the monoid of Euler reals determined by $T$ coincides with $R$.
\end{example}

\begin{remark}[On the complex Gaeta topos.]\label{RemOnSiteOfConnectedObjectsWithPoint}
Essentially by the Basis Theorem, every object in $\Aff_k$ is a finite coproduct of connected objects. It follows that the Gaeta topos ${\calG_k}$ is a equivalent the topos of presheaves on the subcategory of connected objects in $\Aff_k$.
If ${k = \Complex}$, every connected object in $\Aff_k$ has a point, so the Gaeta topos of $\Complex$ is pre-cohesive over $\Set$.
We make this remark because the experience of the Complex Gaeta topos will lead us to consider other sites of connected objects that have a point.
\end{remark}

The similarity  between the intuitions about $T$ and about the object $D$ in SDG is no accident. To give a rigorous comparison we prove the  following  result which has surely been known since the mid 60's although I don't think it appears explicitly in \cite{Lawvere79} or in \cite{KockSDG2ed}.
An `external' form of the result is mentioned  in the paragraph before Section~4 in \cite{Rosicky1984}. The present form is suggested in the last paragraph of p.~{250} of \cite{Lawvere2011} and  it  is stated explicitly just before Corollary~{5.5} in  \cite{CockettCruttwell2014}.

\begin{proposition}[KL implies E.]\label{PropKLimpliesEuler} 
If $R$ is a ring of line type then the multiplicative part of $R$  is the monoid of Euler reals determined by ${\aunit : \final \rightarrow D}$.
\end{proposition}
\begin{proof} 
The  map ${R \times R \rightarrow R^D}$ used in  KL axiom makes the inner  triangle below commute
\[\xymatrix{
\final \times R \ar[d] \ar[r]^-{\aunit \times R} & R\times R \ar[rd]_-{\pr_0} \ar[r]^-\cong & R^D \ar[d]^-{ev_\aunit} && 
  R \ar[d]_-{\twopl{!}{id}}^-{\cong} \ar[rr]^-j &&  D^D \ar[d]^-{i^D}  \\
\final \ar[rr]_-{\aunit} && R && 
 \final \times R \ar[r]_-{\aunit\times R} & R\times R \ar[r]_-{\cong} & R^D
}\]
(where $\pr_0$ is the obvious product projection) and the supplementary inner polygon is a pullback so the left rectangle above  is a pullback. Also, if we let ${i : D \rightarrow R}$ be the evident subobject then the square on the right above commutes where the top map $j$  is the transposition of restricted multiplication ${R \times D \rightarrow D}$; so the rectangle is a pullback. Stacking the pullbacks we obtain that ${R \rightarrow D^D}$ is the inverse image of ${\aunit : \final \rightarrow R}$ along 
${(ev_0) (i^D) : D^D \rightarrow R^D}$. As the square on the left  below commutes
\[\xymatrix{
D^D \ar[d]_-{ev_0} \ar[r]^-{i^D} & R^D \ar[d]^-{ev_0} && 
  R \ar[d] \ar[r] & D^D \ar[d]^-{ev_0} \\
D \ar[r]_-i & R &&
  \final \ar[r]_-{\aunit} & D
}\]
and ${\aunit : \final \rightarrow R}$ factors through ${i : D \rightarrow R}$, it follows that the right square above is a pullback.

It remains to check that $j$ preserves multiplication, in other words, that the diagram on the left below commutes
\[\xymatrix{
R \times R \ar[d]_-{\cdot} \ar[r]^-{j\times j} & D^D \times D^D  \ar[d]^-{\circ} && 
   R \times R \times D \ar[d]_-{\cdot \times D} \ar[rr]^-{R \times \cdot} && R \times D \ar[d]^-{\cdot} \\
R \ar[r]_-j & D^D &&
   R \times D \ar[rr]_-{\cdot} && D
}\]
but, via transposition,  this reduces to the commutativity of the diagram on the right above.
\end{proof}

\begin{remark}[Alternative proof of \ref{PropKLimpliesEuler}.]
For comparison we sketch below a proof using generalized elements suggested by A.~Kock (private communication).
The KL-axiom implies that maps ${D \rightarrow R}$ may be identified with affine endomaps ${\epsilon \mapsto a + b\epsilon}$ with ${a , b \in R}$.
If we denote such a map by ${(a, b)}$ then  it follows that ${(a,b) \circ (c, d) = (a + bc, bd)}$.
Such an affine map takes $D$ into $D$ if and only if ${(a + b\epsilon)^2 = 0}$ for every ${\epsilon \in D}$.
This is equivalent to the conjunction of ${a^2 = 0}$ and, for all ${\epsilon \in D}$, ${2ab \epsilon = 0}$; and we may cancel the universally quantified $\epsilon$. So we have identified the monoid of endos of $D$ as that of  affine maps ${(a, b)}$ such that ${a^2 = 0}$ and ${2ab =0}$.
Such an affine map preserves $0$ if and only if ${a = 0}$. In other words, the  monoid of Euler reals may be identified with that of affine maps of the form ${(0, b)}$ with ${b\in R}$ and, clearly, ${(0, b) \circ (0, c) = (0, bc)}$.
\end{remark}

In any case, regardless of the preferred style of proof, Proposition~\ref{PropKLimpliesEuler} shows, roughly speaking,  that  the `radical' and the `conservative' versions of SDG are compatible.

\begin{example}[Affine $\Cinfty$-schemes.]\label{ExAffineCinftySchemes}
Let $\Aff_{\Cinfty}$ be the opposite of the category of finitely generated $\Cinfty$-rings.
As in the case of algebras over a field, for $A$ a finitely generated $\Cinfty$-ring we let ${\Spec A}$ be the corresponding object in $\Aff_{\Cinfty}$.
It is well-known that $\Spec(\Cinfty(\Real))$ is a ring of line type in $\Aff_{\Cinfty}$, so its underlying multiplicative monoid is a monoid of Euler reals by Propostion~\ref{PropKLimpliesEuler}.
\end{example}

Recall that, in an extensive category, an object is said to be {\em connected} if it has exactly two complemented subobjects.
For instance, the rings of line type in Examples~\ref{ExAffineKschemes} and \ref{ExAffineCinftySchemes} are connected.
On the other hand:

\begin{example}[Rings of line type need not be connected.]
It is known that the ring $R$ of line type  in the Weil topos determined by a field $k$  is a coproduct indexed by  $k$.
See \cite[Exercise~{III.11.4}]{KockSDG2ed} pointing at the existence of non-constant endomorphisms $f$ of $R$ such that ${f' = 0}$.  
See also  \cite[Proposition~{6.3}]{MMlevelEpsilon} where it is evident that, in the complex Weil topos, the set of points of $R$ is isomorphic to its set of connected components.
\end{example}

(A different approach to connectedness is that in \cite[Section III.3]{MoerdijkReyesBook} using internal topological spaces in well-adapted models of SDG, but we will not deal with that here.)

Even assuming that $R$ is connected, the subobject  ${U \rightarrow R}$  of invertibles  may or may not  be connected.

\begin{example}[$U$ may be connected.]\label{ExConnectedU}
For any field $k$, ${ U = \Spec(k[x^{-1}])}$ in ${\Aff_k}$, which is connected.
We stress that $k$ need not be algebraically closed. In particular, $U$ is connected in  ${\Aff_{\Real}}$.
This may appear counter-intuitive but, in some sense,  Algebraic Geometry over a field $k$ is not just about $k$ but also about its finite extensions so, intuitively, we might expect that $\Aff_k$ displays some traits of the separable or algebraic closure of $k$.
\end{example}

In contrast, consider the following.

\begin{example}[$U$ may be disconnected.]\label{ExDisconnectedU}
It follows from \cite[Proposition~{III.6.7}]{KockSDG2ed} that the embedding of manifolds into $\Aff_{\Cinfty}$ preserves the pullback
\[\xymatrix{
(-\infty,0) + (0,\infty) \ar[d] \ar[r] & \final \ar[d]^-{1} \\
\Real \times \Real \ar[r]_-{\cdot} & \Real
}\]
and, since the embedding sends  $\Real^2$ to $R^2$ and $\Real$ to $R$, we may conclude that 
\[ U = \Spec(\Cinfty(-\infty,0) \times \Cinfty(0,\infty)) = \Spec(\Cinfty(-\infty,0)) + \Spec( \Cinfty(0,\infty)) \]
 which is not connected in the extensive $\Aff_{\Cinfty}$. 
\end{example}

In order to continue our discussion it is convenient to recall the following restricted version of \cite[Definition~{0.2}]{Yetter87}.

\begin{definition}\label{DefTdiscrete} 
{\em 
An object $X$  in $\calE$ is {\em $T$-discrete} if ${ev_0 : X^T \rightarrow X}$ is an isomorphism. 
}
\end{definition}
 
To capture the idea of a `discrete space of connected components', the suggestion in \cite{Lawvere2011} is to use  a full  reflective subcategory of $\calE$ of $T$-discrete objects and such that the left adjoint    preserves finite products. 
We will not make emphasis  on $T$-discrete objects. We will simply assume that we have a reflective subcategory ${\calS \rightarrow \calE}$ whose left adjoint ${\pi_0 : \calE \rightarrow \calS}$ preserves finite products.
Given such a reflective subcategory  we may say that an object $X$ in $\calE$  is {\em connected} if ${\pi_0 X = \final}$.

\begin{example}[Decidable affine $k$-schemes.]\label{ExDecidableAffineSchemes}
For a fixed base field $k$, the full subcategory ${\Dec(\Aff_k) \rightarrow \Aff_k}$ of decidable objects has a finite-product preserving left adjoint 
$\pi_0$ \cite[I,\S 4, n$^{\textnormal{o}}$ 6]{DemazureGabriel}. In terms of algebras, the adjoint sends a finitely generated $k$-algebra to its largest separable subalgebra. It follows that $R$ is connected in $\Aff_k$ (i.e. $\pi_0 R = \final$).
\end{example}

Example~\ref{ExDecidableAffineSchemes} may be lifted to toposes as in \cite{MenniExercise}.
Again, see \cite{Menni2021a} for analogous examples in the context of algebras with idempotent addition.

To recapitulate, let $\calE$ be an extensive  category with finite limits and let ${\calS \rightarrow \calE}$ be a full subcategory  with finite-product preserving left adjoint $\pi_0$. 
Let  ${0 : \final \rightarrow T}$ be a pointed object in $\calE$  with exponentiable  $T$ and let  ${R}$ be the associated monoid of Euler reals.

\begin{proposition}[{\cite[Proposition~1]{Lawvere2011}}]\label{PropLawvereEuler} If $R$ is connected then, for every $X$ in $\calE$, ${\pi_0  ev_0 : \pi_0 (X^T) \rightarrow  \pi_0 X}$ is an isomorphism. If $T$ is connected and ${R \rightarrow T^T}$ has a retraction, the converse holds.
\end{proposition}
\begin{proof}
Internal composition provides ${X^T}$ with a `pointed' action ${X^T \times T^T \rightarrow X^T}$ and it is straightforward to check that it restricts to a pointed action ${X^T \times R \rightarrow X^T}$.
As $\pi_0$ preserves finite products it sends pointed actions (of $R$) in $\calE$ to pointed actions (of ${\pi_0 R}$) in $\calS$.
If $R$ is connected then ${0 = 1}$ in the  monoid ${\pi_0  R}$ so the result follows.
In more detail, the condition saying that ${1 \in R}$ acts as the identity  means that the triangle on the right below commutes.
`Pointedness' means that the left square below commutes
$$\xymatrix{
X^T \times \final \ar[d]_-{\pr_0} \ar[rr]^-{id \times 0} && X^T \times R \ar[d] &&   X^T \times \final \ar[ll]_-{id \times 1} \ar[lld]^-{\pr_0} \\
X^T \ar[r]_-{ev_0} & X \ar[r] & X^T
}$$
where ${X \rightarrow X^T}$ is the canonical section of $ev_0$, that is, the transposition of the projection ${X\times T\rightarrow X}$. As   ${\pi_0}$ preserves finite products, the diagram below
$$\xymatrix{
\pi_0 (X^T) \times \final \ar[d]_-{\pr_0} \ar[rr]^-{id \times \pi_0 0} &&\pi_0 (X^T) \times \pi_0 R \ar[d] &&   \pi_0 (X^T) \times \final \ar[ll]_-{id \times \pi_0 1} \ar[lld]^-{\pr_0} \\
\pi_0 (X^T) \ar[r]_-{\pi_0 ev_0} & \pi_0 X \ar[r] & \pi_0 (X^T) 
}$$
commutes in $\calS$.
If ${R}$ is connected (i.e. ${\pi_0 R = \final}$) then the following diagram commutes
$$\xymatrix{
\pi_0 (X^T) \ar[d]_-{\pi_0 ev_0} \ar[rd]^{id} \\
 \pi_0 X \ar[r] & \pi_0 (X^T)
}$$ 
so the retraction ${\pi_0 ev_0 : \pi_0 (X^T) \rightarrow \pi_0 X}$ is an isomorphism.

For the converse we have that  ${\pi_0 (T^T) \cong \pi_0 T = \final}$ by hypothesis.
That is, ${T^T}$ is connected and, since retracts of connected objects are connected, the result follows.
\end{proof}

In other words, under mild assumptions, $R$ is connected if and only if, for every space $X$, the tangent bundle of $X$ has the same connected components as $X$.  (Notice that if $R$ is a ring of line type then  the subobject ${R \rightarrow D^D}$  appearing in the proof of Proposition~\ref{PropKLimpliesEuler}  has a retraction needed to define derivatives.)

Rings of line type have always been commutative.
Concerning monoids of Euler reals \cite{Lawvere2011} says that ``To justify that commutativity seems difficult, though intuitively
it is related to the tinyness of $T$, in the sense that even for slightly larger infinitesimal
spaces, the (pointed) endomorphism monoid is non-commutative". 

Also, \cite{Lawvere2011} briefly discusses  two ways to insure that monoids of Euler reals  have a unique addition.
One via Integration, the other via trivial Lie algebras.
The first one is to consider the subobject ${\Phi(X) = \mathrm{Hom}_R(R^X , R) \rightarrow R^{(R^X)}}$ of the functionals ${\phi}$ such that ${\phi (\lambda f) = \lambda (\phi f)}$ for every ${\lambda \in R}$, and then require that ${\Phi(\final + \final) = (\Phi \final)^2}$, so that addition ``emerges as the unique homogeneous map ${R \times R \rightarrow R}$ which becomes the identity when restricted to both $0$-induced axes ${R \rightarrow R\times R}$.''  The second one is to consider the kernel ${\mathrm{Lie}(R) \rightarrow R^T}$ of ${ev_0 : R^T \rightarrow R}$ which has a binary operation that may be called addition; ``the space of endomorphisms of ${\mathrm{Lie}(R)}$ for that operation is a rig that contains
(the right action of) $R$ as a multiplicative sub-monoid, so that if we postulate that
$R$ exhausts the whole endomorphism space, then $R$ inherits a canonical addition''.

Assuming that the monoid $R$ underlies a ring structure,   the simple result below is applied  to derive, from a subgroup of $R$,  a pre-order on $R$, meaning  a subrig $M$ of `non-negative quantities'.

\begin{proposition}[{\cite[Proposition~2]{Lawvere2011}}]\label{Prop2inLawvere2011}
Given a subobject  ${P\subseteq K}$ of a rig $K$, let 
\begin{center}
${A = \{ a \in K \mid a +  P \subseteq P\} \subseteq K}$ \quad and \quad ${M = \{\lambda \in K \mid \lambda A \subseteq A\}\subseteq K}$. 
\end{center}
Then ${A \subseteq K}$ is an additive submonoid and so ${M \subseteq R}$ is a subrig.
If ${1 \in A}$ then ${M \subseteq A}$.
If ${P \subseteq R}$ is a multiplicative subgroup then ${P \subseteq M}$.
\end{proposition}

For example, if ${P \subseteq \Complex}$ is the subobject of invertible elements then ${A = \{ 0 \}}$ and ${M = \Complex}$.
On the other hand, if ${P = (0, \infty) \subseteq \Real}$ then ${M = A = [0, \infty) \subseteq \Real}$.

\begin{lemma}\label{LemAigualM} 
For ${P, M, A \subseteq K}$ as in Proposition~\ref{Prop2inLawvere2011} the following hold:
\begin{enumerate}
\item If  $K$ is a ring, ${1 \in A}$  and ${-1 \in M}$  then ${0 \in P}$.
\item If ${1 \in A}$ then, the inclusion ${M \subseteq A}$ is an equality of subobjects if and only if $A$ is closed under multiplication.
\end{enumerate}
\end{lemma}
\begin{proof}
 If ${1 \in A}$ then ${M \subseteq A}$ by Proposition~\ref{Prop2inLawvere2011} so, if ${-1 \in M}$ then ${-1 \in A}$ and hence ${-1 + 1  =  0 \in P}$. 
One direction of the second item is trivial.
For the other, assume that $A$ is closed under multiplication. 
Then, for every ${a \in A}$, ${a A \subseteq A}$ so ${a \in M}$.
\end{proof}

Let ${U \rightarrow R}$ be the subobject of invertible elements and  let the following square
\[\xymatrix{
U_+ \ar[d] \ar[r] & \pi_0 \final \ar[d]^-{\pi_0 1} \\
U \ar[r] & \pi_0 U
}\]
be a pullback, where ${1 : \final \rightarrow U}$ is the unit of $R$ as a subobject of $U$.
Since $U$ is a subgroup of $R$, and $\pi_0$ preserves products, ${\pi_0 U}$ is also a group and ${U_+ \rightarrow U}$ is the kernel of the group morphism ${U \rightarrow \pi_0 U}$.

The map ${\pi_0 1 : \final \rightarrow \pi_0 U}$ may be an isomorphism and, in this case, ${U_+ \rightarrow U}$ is also an isomorphism, of course.
We are mainly interested in contexts where this is not the case.

\begin{definition}\label{DefBidirectional}
{\em 
The monoid $R$ is {\em bi-directional} if  ${\pi_0 U = \final + \final}$.
}
\end{definition}

So, assuming that the monoid $R$ of Euler reals underlies a ring, we may apply Proposition~\ref{Prop2inLawvere2011} to the subgroup ${U_+ \subseteq R}$ in order to obtain a subrig ${M \subseteq R}$.  The `real'  intuition suggests that it is not unnatural to require or expect that ${1  + U_+ \subseteq U_+}$; in other words, ${1 \in A \subseteq R}$.  Then the first item of Lemma~\ref{LemAigualM} implies that ${-1  \not\in M}$, unless ${0 = 1}$ in $R$.

Readers are invited to compare the above with the efforts in \cite{KockSDG2ed,MoerdijkReyesBook} to prove that the pre-orders defined on certain rings of line type there are compatible with the ring operations. We should also remark that some of those efforts will play a role below when the time comes to give a simple description of the subrigs ${M \rightarrow R}$ derived radically synthetically.

\section{A  bi-directional monoid of Euler reals}
\label{SecFirstExample}

Let ${p : \calE \rightarrow \calS}$ be a geometric morphism.
Recall that $p$ is {\em hyperconnected} if ${p^* : \calS \rightarrow \calE}$ is fully faithful and the counit  $\beta$  of ${p^* \dashv p_*}$ is monic. Intuitively, $\calE$ is a category of spaces, ${p^*}$ is the full subcategory of discrete spaces, and the right adjoint ${p_* : \calE \rightarrow \calS}$ sends a space to the associated discrete space of points. For a space $X$, the monic ${\beta_X : p^*(p_* X) \rightarrow X}$ may be thought of as the discrete subspace of points.

In perspective, we may say that $p$ is  {\em pre-cohesive} if it is hyperconnected, $p^*$ is cartesian closed and ${p_*}$ preserves coequalizers.
It follows  \cite{Menni2021} that ${p_*}$ has a (necessarily fully faithful) right adjoint $p^!$ and that ${p^*}$ has a left adjoint ${ p_!  : \calE \rightarrow \calS}$ that preserves finite products. Hence, $p$ is pre-cohesive if and only if ${p^* \dashv p_*}$ extends to a string of adjoints
${p_!\dashv p^* \dashv p_* \dashv p^!}$ such that ${p^*}$ is fully faithful, the counit of ${p^* \dashv p_*}$ is monic and the leftmost adjoint $p_!$ preserves finite products. It follows that the reflector ${\sigma_X : X \rightarrow p^* (p_! X) = \pi_0  X}$ is epic.
Sometimes we say  that $\calE$ is pre-cohesive over $\calS$.

For any pre-cohesive $p$, the fact that the  canonical  ${p_* \sigma : p_* \rightarrow p_* p^* p_! \cong p_!}$ is an epimorphism is consistent with the intuition that every connected component  has a point.
In particular, if  ${p_* X = \final }$ then ${p_! X = \final}$. In other words, if the space $X$ has a unique point then it is connected.

The purpose of this section is to build a pre-cohesive topos ${p : \calE \rightarrow \Set}$ and a (necessarily connected) object $T$ in $\calE$ such that ${p_* T = \final}$, and such that the resulting monoid $R$ of Euler reals is bi-directional.
Moreover, it will be evident from the construction that $R$ underlies a ring of line-type.

Pre-cohesion will follow from the next result borrowed from \cite{Johnstone2011}. (See also \cite[Proposition~{2.10}]{MenniExercise} for a statement consistent with our terminology.)

\begin{proposition}\label{PropJohnstone} Let $\calC$ be a small category with terminal object. Then ${\Psh{\calC} \rightarrow \Set}$ is pre-cohesive if and only if every object in $\calC$ has a point.
\end{proposition}

Now recall Remark~\ref{RemOnSiteOfConnectedObjectsWithPoint} and let ${\calC \rightarrow \Aff_{\Cinfty}}$ be the full subcategory of connected objects that have a point.
The category $\calC$ is essentially small but contains many objects of interest.
It certainly has a terminal object. Also,  germ-determined $\Cinfty$-rings have a copoint almost by definition (\cite[Exercise~{III.7.1}]{KockSDG2ed}), so every non-trivial germ-determined  finitely-generated $\Cinfty$-ring  with exactly two idempotents determines an object in $\calC$.

In particular, for any manifold $M$,  ${\Cinfty(M)}$ is finitely-presentable and germ-determined (in fact, point-determined) by \cite[Theorem~{III.6.6} and Corollary~{III.5.10}]{KockSDG2ed}.
Moreover, if $M$ is  connected then ${\Cinfty(M)}$ has exactly two idempotents by the Intermediate Value Theorem.
Altogether, for any connected manifold $M$, ${\Cinfty(M)}$ determines an object of $\calC$. 

Also, every Weil algebra (over $\Real$) is a finitely presentable $\Cinfty$-ring by \cite[Proposition~{III.5.11}]{KockSDG2ed} and hence determines an object in $\calC$.

Let $T$ be the object in $\calC$ determined by the Weil algebra $\Real[\epsilon]$.
Of course, it has a unique point, and we denote it by ${0 : \final \rightarrow T}$.

\begin{lemma}\label{LemProductsWithT}  For any $X$ in $\calC$, the product ${X \times T}$ in $\Aff_{\Cinfty}$ is also in $\calC$.
\end{lemma}
\begin{proof}
The forgetful functor ${\inftyRing \rightarrow \Real/\Ring}$ reflects coproducts with Weil algebras by \cite[Theorem~{III.5.3}]{KockSDG2ed}.
  More precisely, if $A$ is a $\Cinfty$-ring and  $W$ is a Weil algebra then there is a unique $\Cinfty$-ring structure on  ${A \otimes_{\Real} W}$ extending its $\Real$-algebra structure  such that ${A \rightarrow A\otimes_{\Real} W \leftarrow W}$ is a coproduct in $\inftyRing$.
In particular,  ${A\otimes_{\infty} \Real[\epsilon] = A[\epsilon]}$. So, if $A$ is finitely generated, or has exactly two idempotents, or has a copoint, then so does ${A[\epsilon]}$. Therefore, if $A$ determines an object in $\calC$ then so does the coproduct ${A\otimes_{\infty} \Real[\epsilon]}$. \end{proof}

By the remarks above, the object ${R = \Spec(\Cinfty(\Real))}$ is in the subcategory ${\calC \rightarrow \Aff_{\Cinfty}}$.
The product 
\[ R\times R = \Spec(\Cinfty(\Real) \otimes_{\infty} \Cinfty(\Real)) = \Spec(\Cinfty(\Real^2)) \]
 is also in $\calC$, so $R$ is a ring object in $\calC$.

\begin{proposition}\label{Propbi-directional} 
The canonical ${\Psh{\calC} \rightarrow \Set}$ is pre-cohesive and the monoid of Euler reals determined by $T$ in $\Psh{\calC}$ is bi-directional. Moreover, that monoid coincides with (the multiplicative part of) the ring $R$ which is  connected and satisfies the KL-axiom. 
\end{proposition}
\begin{proof}
By definition, the category $\calC$ has a terminal object and every object in it has a point, so   ${\Psh{\calC} \rightarrow \Set}$ is pre-cohesive by Proposition~\ref{PropJohnstone}.

The quotient ${\Cinfty(\Real) \rightarrow \Real[\epsilon]}$ in $\inftyRing$ determines a monomorphism ${T\rightarrow R}$ in $\calC$. In fact, it is the subobject of elements of square zero.  Moreover, by Lemma~\ref{LemProductsWithT} we can repeat the usual proof that the KL-axiom holds.
Indeed, for any ${\Spec A}$ in $\calC$, 
\[ \Psh{\calC}(\Spec A, R^T) \cong \Psh{\calC}((\Spec A) \times T, R) \cong \Psh{\calC}(\Spec ({A[\epsilon]}), R) \cong \]
\[ \inftyRing(\Cinfty(\Real), A[\epsilon]) \cong A\times A \cong \inftyRing(\Cinfty(\Real), A) \times  \inftyRing(\Cinfty(\Real), A) \cong \]
\[ \cong \Psh{\calC}(\Spec A, R) \times \Psh{\calC}(\Spec A, R) \cong \Psh{\calC}(\Spec A, R\times R) \]
 as usual. Proposition~\ref{PropKLimpliesEuler}  implies that the monoid of Euler reals determined by $T$ coincides with the multiplicative part of $R$.
 Moreover, $R$  is connected because it is representable.

 As observed in Example~\ref{ExDisconnectedU}, the  subobject ${ U \rightarrow R}$ of invertibles  in $\Aff_{\Cinfty}$ coincides with  ${\Spec(\Cinfty(-\infty,0)) + \Spec(\Cinfty(0, \infty))  \rightarrow \Spec \Real = R}$.
 The domain of this subobject is not an object in $\calC$, but the summands are.
 Then  
 $${U = \Spec(\Cinfty(-\infty,0)) + \Spec(\Cinfty(0, \infty)) }$$
  in $\Psh{\calC}$.
So ${p_! U  = p_! [\Spec( \Cinfty(-\infty,0)) + \Spec(\Cinfty(0, \infty))]  = \final + \final}$ because $U$ in $\Psh{\calC}$ is a coproduct of two representables.
  \end{proof}

It is clear  from the proof of Proposition~\ref{Propbi-directional}  that ${U_+ = \Spec(\Cinfty(0,\infty))}$ in $\Psh{\calC}$.
Then  ${1 + U_+  \subseteq U_+}$  and so, by the remarks following Lemma~\ref{LemAigualM}, we obtain a subrig ${M \subseteq R}$ inside the complement of ${-1 : \final \rightarrow R}$. I have not found an illuminating expression of this subrig though. In the next section we consider a smaller topos, also with a bi-directional $R$, but where ${M \subseteq R}$ has a simple description.

\section{W-determination and the Positiv-stellen-satz}
\label{SecW}

In this section we construct another bi-directional ring of line type but in a topos where the induced pre-order is easier to describe explicitly in terms of the site.

\begin{def}\label{DefWdetermined} 
{\em 
A $\Cinfty$-ring $A$ is {\em W-determined}  if  the family of maps ${A \rightarrow W}$ in $\inftyRing$ with Weil codomain is jointly monic.
}
\end{def}

W-determined $\Cinfty$-rings are called {\em near-point determined} in \cite{MoerdijkReyesBook}.

Let ${\calC_{W} \rightarrow \calC}$ be the full subcategory induced by the objects in $\calC$ that are W-determined as $\Cinfty$-rings.
It follows from the discussion following \cite[Theorem~{III.9.4}]{KockSDG2ed} that every connected manifold with boundary determines an object in $\calC_W$. Moreover, this assignment is functorial.

\begin{proposition}\label{PropAnotherbi-directional}
 The canonical ${\Psh{\calC_W} \rightarrow \Set}$ is pre-cohesive and the monoid of Euler reals induced  by $T$ in $\Psh{\calC}$ is bi-directional. Moreover, this monoid coincides with (the multiplicative part of) the ring $R$ which is  connected and satisfies the KL-axiom. 
\end{proposition}
\begin{proof}
It is enough to check that ${\calC_{W}}$ has all the properties needed to mimic  Proposition~\ref{Propbi-directional}. 
It certainly contains all the objects induced by connected manifolds and also those induced by Weil algebras.
In particular, it contains the object $T$.
So the main problem is to extend Lemma~\ref{LemProductsWithT} by showing that ${\calC_W}$ is closed under products with $T$, but this follows from \cite[Proposition~{I.4.6} and Lemma~{II.1.15}]{MoerdijkReyesBook}.
\end{proof}

Let  ${\Cinfty{[0,\infty)} = \Cinfty(\Real)/I}$ where ${I \subseteq \Cinfty(\Real)}$ is the ideal of functions that vanish on ${[0, \infty) \subseteq \Real}$.

\begin{proposition}[Positiv-Stellen-Satz.]\label{PropPositivStellenSatz}
Let $m$ be a finite set and let ${q :  \Cinfty(\Real^m) \rightarrow A}$ be a regular epimorphism in $\inftyRing$ with W-determined codomain.
If we let   ${J \subseteq \Real^m}$ be the kernel of  $q$ then, for  any smooth ${g : \Real^m \rightarrow \Real}$, the following are equivalent:
\begin{enumerate}
\item The restriction of the smooth ${g : \Real^m \rightarrow \Real}$ to ${Z(J) \subseteq \Real^m}$ factors through ${[0,\infty) \subseteq \Real}$, where ${Z(J) = \{ x\in \Real^m \mid (\forall f \in J) (f x = 0)\}}$.
\item The composite 
\[\xymatrix{\Cinfty(\Real) \ar[r]^-{\Cinfty g} & \Cinfty(\Real^m) \ar[r]^-q & A}\] 
factors through ${\Cinfty(\Real) \rightarrow \Cinfty[0,\infty)}$.
\end{enumerate}
\end{proposition}
\begin{proof}
Use the proof of \cite[Lemma~{III.11.4}]{KockSDG2ed}.
\end{proof}

The $\Cinfty$-ring ${\Cinfty{[0,\infty)}}$  is  W-determined  by the remarks in \cite[p.185]{KockSDG2ed}.
The corresponding object in $\calC_W$ will  be denoted by $H$.
So the quotient ${\Cinfty(\Real) \rightarrow \Cinfty[0,\infty)}$ induces a  monomorphism ${H \rightarrow R}$  in $\calC_W$.

Let ${\Gamma = \calC_W(\final, -) : \calC_W \rightarrow \Set}$ be the usual `points' functor.

\begin{corollary}\label{CorSpatialPositivStellenSatz}
For every ${v : X \rightarrow R}$ in $\calC_W$, $v$ factors through the subobject  ${H \rightarrow R}$ if and only if ${\Gamma v : \Gamma X \rightarrow \Gamma R = \Real}$ factors through ${[0,\infty) \rightarrow \Real}$.
\end{corollary}
\begin{proof}
The object $X$ equals ${\Spec A}$ for some finitely generated and W-determined $\Cinfty$-ring $A$.
So there is a finite set $m$ and  a regular epimorphism ${q : \Cinfty(\Real^m) \rightarrow A}$. 
The map $v$ corresponds to a map ${\overline{v} : \Cinfty(\Real) \rightarrow A}$ and, since the domain of this map is projective, there exists a smooth ${g : \Real^m \rightarrow  \Real}$ such that the following triangle
\[\xymatrix{
\Cinfty(\Real) \ar@(d,l)[rrd]_-{\overline{v}} \ar[rr]^-{\Cinfty g} && \Cinfty(\Real^m) \ar[d]^-q \\
&& A
}\]
commutes in $\inftyRing$. 
By Proposition~\ref{PropPositivStellenSatz},  ${\overline{v}}$ factors through ${\Cinfty(\Real) \rightarrow \Cinfty{[0,\infty)}}$ if and only if  the restriction of $g$ to ${Z(J) \subseteq \Real^m}$ factors through ${[0,\infty) \subseteq \Real}$ where $J$ is the kernel of $q$. But this holds if and only if the left-bottom composite below
\[\xymatrix{
Z(J) = \Gamma X \ar[d] \ar@{.>}[r] & [0,\infty) \ar[d] \\
\Gamma R^m =\Real^m \ar[r]_-g & \Real
}\]
factors through the right subobject as depicted above, so the result follows.
\end{proof}

For a connected manifold $M$ we let ${\overline{M} = \Spec (\Cinfty(M))}$. 
To complete our calculation we need another standard result that we reformulate as follows.

\begin{lemma}\label{LemProductsWithH}
 For any connected manifold $M$, the product of $\overline{M}$ and ${H}$ exists in $\calC_W$ and it corresponds to ${\Cinfty(M\times \Real)/J}$ where $J$ is the ideal of functions that vanish on ${M \times [0,\infty)}$.
\end{lemma}
\begin{proof}
This follows from the proof of \cite[Theorem~{III.9.5}]{KockSDG2ed} (and the fact that, for connected $M$, ${\Cinfty(M\times \Real)/J}$ has exactly two idempotents).
\end{proof}

Recall that ${U \rightarrow R}$ is the subobject of invertibles, that ${U_+ \rightarrow U}$ is the kernel of ${U \rightarrow \pi_0 U}$ and that  ${M \subseteq R}$ is the subrig determined by bi-directionality of $R$.

\begin{theorem} In the topos $\Psh{\calC_W}$, the subrig ${M \rightarrow R}$ coincides with ${H \rightarrow R}$.
\end{theorem}
\begin{proof}
Consider first  
${  A = \{ a \in R \mid a + U_+ \subseteq U_+ \} = \{ a \in R \mid (\forall u \in U_+)(a + u \in U_+) \} \rightarrow R}$. 
To check that ${H \leq A}$ is enough, by Lemma~\ref{LemProductsWithH}, to show that the composite 
\[\xymatrix{
   U_+  \times H \ar[r] & R \times R \ar[r]^-{+} & R
}\]
factors through ${U_+ \rightarrow R}$, but this follows from the embedding of connected manifolds with boundary in $\calC_W$ and the fact that
the composite
\[\xymatrix{
(0,\infty) \times [0,\infty)  \ar[r] & \Real \times \Real \ar[r]^-{+} & \Real
}\]
factors through ${ (0,\infty)  \subseteq \Real}$ there.

To prove that ${A \leq H}$ let $C$ in $\calC_W$ and let ${\gamma :  C \rightarrow R}$ in $\Psh{\calC_W}$ factor through ${A \rightarrow R}$.
It follows that, for every ${c : \final \rightarrow C}$ and ${r : \final \rightarrow U_+}$, the composite
\[\xymatrix{
\final \ar[rr]^-{\twopl{r}{c}} && U_+ \times  C \ar[rr]^-{U_+ \times \gamma} && U_+ \times R \ar[r]^-{+} & R 
}\]
factors through ${U_+ \rightarrow R}$. In other words, the composite
\[\xymatrix{
(0,\infty) \times \Gamma C  \ar[rr]^-{(0,\infty) \times \Gamma \gamma} && (0,\infty) \times \Gamma R  = (0,\infty) \times \Real \ar[r]^-{+} & \Real
}\]
factors through ${(0,\infty) \rightarrow \Real}$. Then ${\Gamma \gamma : \Gamma C \rightarrow \Gamma R = \Real}$ factors through ${[0,\infty) \rightarrow \Real}$. So $\gamma$ factors through ${H \rightarrow R}$ by Corollary~\ref{CorSpatialPositivStellenSatz}.
This completes the proof that ${A = H}$ as subobjects of $R$.

Certainly, ${1 : \final \rightarrow R}$ factors through ${H = A \rightarrow R}$ so, by Lemma~\ref{LemAigualM}, it only remains to check that the subobject ${H \rightarrow R}$ is closed under multiplication. That is, we need  to prove that the composite
\[\xymatrix{ H \times H \ar[r] & R\times R \ar[r]^-{\cdot} & R }\]
 factors through ${H \rightarrow R}$.  So let $C$ in $\calC$ and let ${u, v : C \rightarrow H}$ in $\Psh{\calC_W}$.
 Then the composite
 \[\xymatrix{ \Gamma C \ar[rr]^-{\twopl{\Gamma u}{\Gamma v}}  && \Gamma H \times \Gamma H  = [0, \infty) \times [0,\infty) \ar[r] & \Real\times \Real \ar[r]^-{\cdot} & \Real }\]
clearly factors through ${[0, \infty) \rightarrow \Real}$, because the restricted multiplication ${[0,\infty) \times [0, \infty) \rightarrow \Real}$ does.
Hence, for every $u, v$ as above, the composite
\[\xymatrix{
 C \ar[rr]^-{\twopl{ u}{ v}}  &&  H \times  H   \ar[r] & R\times R \ar[r]^-{\cdot} & R 
} \]
factors through ${H\rightarrow R}$ in $\Psh{\calC_W}$ by Corollary~\ref{CorSpatialPositivStellenSatz}.
Therefore the composite 
\[\xymatrix{H \times H \ar[r] & R\times R \ar[r]^-{\cdot} & R }\]
does factor through ${H \rightarrow R}$ as we needed to prove.
\end{proof}

\bibliographystyle{alpha}   

\end{document}